\documentclass{amsart}
\usepackage{amssymb}
\usepackage{stmaryrd} 
\usepackage{amsmath} 
\usepackage{amscd}
\usepackage{amsbsy}
\usepackage{commath, longtable}
\usepackage{comment, enumerate}
\usepackage{geometry}
\usepackage[matrix,arrow]{xy}
\usepackage{hyperref}
\usepackage{mathrsfs}
\usepackage{color}
\usepackage{mathtools,caption}
\usepackage[table,dvipsnames]{xcolor}
\usepackage{tikz-cd}
\usepackage{longtable}
\usepackage[utf8]{inputenc}
\usepackage[OT2,T1]{fontenc}
\usepackage[normalem]{ulem}
\usepackage{hyperref}
\usepackage[customcolors]{hf-tikz}
\hypersetup{
 colorlinks=true,
 linkcolor=DarkOrchid,
 filecolor=blue,
 citecolor=olive,
 urlcolor=orange,
 pdftitle={IAS project},
 }
\usepackage{booktabs}

\DeclareSymbolFont{cyrletters}{OT2}{wncyr}{m}{n}
\DeclareMathSymbol{\Sha}{\mathalpha}{cyrletters}{"58}

\newcommand{\con}{\mathrm{cong}}

\newcommand{\mat}[1]{\left[ \begin{smallmatrix} #1 \end{smallmatrix} \right]}

\DeclareMathOperator{\Frob}{Frob}
\DeclareMathOperator{\Gal}{Gal}

\DeclareMathOperator{\GL}{GL}

\DeclareMathOperator{\rank}{rank}

\DeclareMathOperator{\Sel}{Sel}

\DeclareMathOperator{\ns}{ns}

\newcommand{\cg}{\cellcolor{gray!20}}
\newcommand{\QQ}{\mathbb Q}
\newcommand{\ZZ}{\mathbb Z}

\newcommand{\bF}{\mathbb F}
\newcommand{\cE}{\mathcal E}

\newcommand{\cO}{\mathcal O}

\newcommand{\cC}{\mathcal{C}}

\newcommand{\fD}{\mathfrak{D}}

\newcommand{\sP}{\mathscr{P}}
\newcommand{\sQ}{\mathscr{Q}}
\newcommand{\pdiv}{\mid\!\mid}

\newcommand{\ta}[1]{\tikzmarkin[fill=gray!20, set border color=gray!20]{#1}(0,-0.35)(-0,0.5)}
\newcommand{\tb}[1]{\tikzmarkend{#1}}

\newcommand{\Z}{{\mathbb{Z}}}
\newcommand{\Q}{{\mathbb{Q}}}

\newcommand{\F}{{\mathbb{F}}}
\newcommand{\trace}{{\mathrm{trace}}}

\newcommand{\Tam}{{\mathrm{Tam}}}
\newcommand{\jr}{\bar\rho_{E_1,E_2,\ell}}

\newtheorem*{theorem*}{Theorem}
\newtheorem{theorem}{Theorem}[section]
\newtheorem{lemma}[theorem]{Lemma}
\newtheorem{proposition}[theorem]{Proposition}
\newtheorem{corollary}[theorem]{Corollary}

\newtheorem{problem}[theorem]{Problem}
\newtheorem{lthm}{Theorem}

\newcommand{\Zl}{\Z_\ell}
\definecolor{Green}{rgb}{0.0, 0.5, 0.0}

\theoremstyle{definition}
\newtheorem{definition}[theorem]{Definition}
\theoremstyle{remark}
\newtheorem{remark}[theorem]{Remark}
\newtheorem*{Remark*}{Remark}
\newtheorem{example}[theorem]{Example}

\makeatletter
\theoremstyle{plain} 
\newtheorem*{intr@thm}{\intr@thmname}

\newtheorem*{c@njecture}{\conjn@name}
\newcommand{\myl@bel}[2]{%
  \protected@write \@auxout {}{\string \newlabel {#1}{{#2}{\thepage}{#2}{#1}{}} }%
  \hypertarget{#1}{}
    } 
\newenvironment{labelledconj}[3][]%
    {
        \def\conjn@name{#2}
        \begin{c@njecture}[{#1}]\myl@bel{#3}{#2}
    }
    {
        \end{c@njecture}
    }
\makeatother

\begin{document}
\title[]{Studying Hilbert's $10^{\text{th}}$ problem via explicit elliptic curves}

\author[D.~Kundu]{Debanjana Kundu}
\address[Kundu]{Department of Mathematics \\ University of British Columbia \\
 Vancouver BC, V6T 1Z2, Canada}
\email{dkundu@math.ubc.ca}

\author[A.~Lei]{Antonio Lei}
\address[Lei]{Department of Mathematics and Statistics\\University of Ottawa\\
150 Louis-Pasteur Pvt\\
Ottawa, ON\\
Canada K1N 6N5}
\email{antonio.lei@uottawa.ca}

\author[F.~Sprung]{Florian Sprung}
\address[Sprung]{School of Mathematical and Statistical Sciences, Arizona State University, 901 S Palm Walk, Tempe AZ 85287-1804, USA}
\email{ian.sprung@gmail.com}

\date{\today}


\begin{abstract}
N.~García-Fritz and H.~Pasten showed that Hilbert's $10^{\text{th}}$ problem is unsolvable in the ring of integers of number fields of the form $\mathbb{Q}(\sqrt[3]{p},\sqrt{-q})$ for positive proportions of primes $p$ and $q$.
We improve their proportions and extend their results to the case of number fields of the form $\mathbb{Q}(\sqrt[3]{p},\sqrt{Dq})$, where $D$ belongs to an explicit family of positive square-free integers. 
We achieve this by using multiple elliptic curves, and replace their Iwasawa theory arguments by a more direct method.
\end{abstract}

\maketitle

\section{Introduction}

\subsection{Historical Remarks}
In 1900, D.~Hilbert posed the following problem:
\begin{labelledconj}{Hilbert's \texorpdfstring{$10^{\text{th}}$}{} Problem}{H10}
``Given a Diophantine equation with any number of unknown quantities and with rational integral numerical coefficients: to devise a process according to which it can be determined in a finite number of operations whether the equation is solvable in rational integers.\footnote{Gegeben eine diophantische Gleichung mit beliebig vielen Unbekannten und mit rationalen ganzzahligen numerischen Koeffizienten: Ein Verfahren entwickeln, nach dem in endlich vielen Operationen bestimmt werden kann, ob die Gleichung in rationalen ganzen Zahlen lösbar ist.}''
\end{labelledconj}

This problem is now known to have a negative answer, i.e., such a general `process' (algorithm) does not exist.
In one sentence, the reason for this negative answer is that the \emph{computably enumerable sets} (whose elements an algorithm can list) in $\mathbb{Z}$ are precisely the \emph{Diophantine} sets over $\mathbb{Z}$ (roughly, `coming from polynomials and integers', see Definition~\ref{defn: Diophantine over a ring}).
For the convenience of the reader, we give a brief sketch of the argument described in \cite[Chapter~5, Section~7]{murty2019hilbert}.

The Diophantine sets of integers can be enumerated -- call them $D_1, D_2, \ldots, D_n, \ldots $ However, not all sets are Diophantine.
For example, Cantor's diagonalization method shows that the set $V=\{n:n\not\in D_n\}$ is not Diophantine.
The set of ordered pairs $(n,x)$ with $n$ an integer and $x\in D_n$ is Diophantine (\emph{Universality Theorem}), i.e., they are characterized as the solutions to $f(n,x, y_1,\ldots, y_m)=0$ for some integers $y_1, \ldots, y_m$ and a fixed polynomial $f$.
Now if a putative algorithm $\mathcal{A}$ could determine whether an arbitrary Diophantine equation has integral solutions, then $\mathcal{A}$ could do this for $f$ and thus test whether $x\in D_n$.
But then $\mathcal{A}$ could test whether $n\in D_n$ and list all those $n$ with $n\not\in D_n$, contradicting that $V$ is \emph{not} computably enumerable.

M.~Davis, H.~Putnam, and J.~Robinson made significant progress towards answering \ref{H10} in the 1950's and 1960's, showing that the computably enumerable sets are exponential Diophantine\footnote{A set $S$ of ordered $n$-tuples is called \emph{exponential Diophantine} if there exists a polynomial $f(x_1, \ldots, x_n, u_1, \ldots, u_m, v_1, \ldots, v_m, w_1, \ldots, w_m)$ such that $(x_1, \ldots, x_n)\in S$ if and only if there exists $u_i, v_i, w_i$ with $f=0$ and $u_i=v_i^{w_i}$ for each $i$.} \cite{davis1961decision}.
The final piece of the puzzle, showing that the exponential function is Diophantine, was proved by Y.~Matiyasevi\v{c} in 1970 \cite{matijasevic1970enumerable}.
The proof involved showing that the Fibonacci numbers are Diophantine, and allowed for the conclusion that no algorithm as solicited by Hilbert exists:

\begin{labelledconj}{MRDP Theorem}{Thm:Matiyasevich}
\ref{H10} has a negative solution.
\end{labelledconj}

J.~Denef and L.~Lipshitz \cite{DL78} generalized \ref{H10} as follows:

\begin{labelledconj}{Denef--Lipshitz Conjecture}{conj}
$\Z$ is a Diophantine subset of the ring of integers of any number field $L$.
\end{labelledconj}

Here are the instances where this conjecture has been resolved.
\begin{enumerate}[(a)]
\item $L$ is totally real or a quadratic extension of a totally real field; see \cite{denef,DL78}.
\item $[L:\Q]=4$, $L$ is not totally real, and $L/\Q$ has a proper intermediate field; see \cite{DL78}.
\item $L$ has exactly one complex place; see \cite{Phe88, Shlapen, videla16decimo}.
\item $L$ is a subfield of one of the extensions mentioned above; see \cite{SS89}.
In particular, it follows from the Kronecker--Weber Theorem that the analogue of \ref{H10} is unsolvable when $L/\Q$ is abelian.
\item If the conjecture holds for a number field $F$, then it holds for certain infinite families of degree $\ell^n$-extensions $L$ of $F$.
More precisely, once $F$ has been chosen, then for all but finitely many primes $\ell$ and all $n\geq1$, the conjecture holds for infinitely many cyclic $\ell^n$-extensions $L$ of $F$; see \cite{MR10, MR18}\footnote{We thank Karl Rubin for patiently clarifying this point.
The main ingredient for deriving this from \cite[Theorem 1.2]{MR18} is that the simple abelian variety can be chosen to be a non-CM elliptic curve.}.  See also \cite{ray} for a recent result on cyclotomic $\Zl$-extensions.
\item\label{gfpresult} $L$ belongs to an explicit family of number fields of the form $\Q(\sqrt[3]{p},\sqrt{-q})$; see \cite{GFP}.
\end{enumerate}

The last two results make use of a link initiated by B.~Poonen between the \ref{conj} and the theory of elliptic curves.
Poonen showed in \cite{Poo02} that extensions of number fields are \emph{integrally Diophantine} (Definition \ref{def: Integrally Diophantine}) if there is an elliptic curve whose rational points in those number fields have rank $1$.
A.~Shlapentokh in \cite{Shlapen} vastly generalized this to arbitrary non-zero rank (\ref{Thm:Shlapentokh}).
B.~Mazur and K.~Rubin showed in \cite{MR10} that if the $2$-torsion part of the Shafarevich--Tate groups of elliptic curves is a square for every number field, then the \ref{conj} holds.
In \cite{MP18}, R.~Murty and H.~Pasten considered this problem from the point of view of analytic aspects of $L$-functions instead, showing that the automorphy conjecture and the rank part of the Birch and Swinnerton–Dyer conjecture would imply it.

The precise version of the most recent result (\ref{gfpresult}) above of N.~García-Fritz and H.~Pasten is:

\begin{theorem}[García-Fritz--Pasten]
\label{thm:GFP}
There are explicit Chebotarev sets (see Definition~\ref{defi: Chebotarev set}) of primes $\sP$ and $\sQ$, of density $\frac{5}{16}$ and $\frac{1}{12}$, such that for all $p\in \sP$ and $q\in \sQ$, the analogue of \ref{H10} is unsolvable for the ring of integers of $L=\Q(\sqrt[3]{p},\sqrt{-q})$.
\end{theorem}


\subsection{Our results}
We give a more direct proof of Theorem~\ref{thm:GFP} without using results from Iwasawa theory and improve the density of both $\sP$ and $\sQ$, as well as proving similar results for new families of extensions.
The key result for this more direct proof is a vanishing theorem of certain Selmer groups, see Theorem~\ref{thm:growth-rank}, which allows us to form the set $\sP$.
Similar to the method utilized in \cite{GFP}, the set $\sQ$ is obtained by seeking rank-one quadratic twists of our auxiliary elliptic curve, which relies crucially on the work of D.~Kriz and C.~Li \cite{krizli}.
By carefully refining the techniques developed in \cite{GFP}, we prove the following improvement of Theorem~\ref{thm:GFP}.
\begin{lthm}[{Theorem~\ref{thm: improve P and Q of GFP}}]
\label{thmA}
There are explicit Chebotarev sets of primes $\sP$ and $\sQ$, of density $\frac{9}{16}$ and $\frac{7}{48}$, such that for all $p\in \sP$ and $q\in \sQ$, the analogue of \ref{H10} is unsolvable for the ring of integers of $L=\Q(\sqrt[3]{p},\sqrt{-q})$.
\end{lthm}

The key ingredient of the proofs of both Theorems~\ref{thm:GFP} and \ref{thmA} is the exhibition of an explicit auxiliary elliptic curve $E$ (of Mordell--Weil rank $0$) such that the Mordell--Weil ranks of $E$ over $\QQ(\sqrt[3]{p})$ and $\QQ(\sqrt{-q})$ are $0$ and $1$, respectively (see \S~\ref{section: strategy} where we review the general strategy employed).
By working with a new auxiliary elliptic curve, we are able to prove a similar result for the following new families of number fields:
\begin{lthm}[{Theorem~\ref{thm: real quadratic one EC two K}}]
\label{thmB}
Let
\[
\fD=\{7, 39, 95, 127, 167, 255, 263, 271, 303, 359, 391, 447, 479, 527, 535, 615, 623, 655, 679, 695\}.
\]
For all $D\in\fD$, there are explicit Chebotarev sets of primes $\sP$ (independent of $D$) and $\sQ_D$, of density $\frac{9}{16}$ and $\frac{1}{12}$ such that for all $p\in \sP$ and $q\in \sQ_D$, the analogue of \ref{H10} is unsolvable for the ring of integers of $L=\Q(\sqrt[3]{p},\sqrt{Dq})$.
\end{lthm}
The families of number fields studied in Theorem~\ref{thmB} are disjoint from the ones studied in Theorem~\ref{thm:GFP} and Theorem~\ref{thmA}.
Furthermore, these number fields are not Galois over $\QQ$ and have exactly two complex places.
In particular, they are not covered in the list of previous known results.

We also prove a modified version of Theorem~\ref{thmB} by working with a pair of auxiliary elliptic curves, where we improve significantly the density of $\sP$, at the expense of a smaller set of $\fD$ and a lower density for the sets $\sQ_D$.

\begin{lthm}[{Theorem~\ref{thm: real quadratic two EC two K}}]
\label{thmC}
Let $D\in\{7,615\}$.
There are explicit Chebotarev sets of primes $\sP$ (independent of $D$) and $\sQ_D$, of density $\frac{103}{128}$ and $\frac{1}{36}$, such that for all $p\in \sP$ and $q\in \sQ_D$, the analogue of \ref{H10} is unsolvable for the ring of integers of $L=\Q(\sqrt[3]{p},\sqrt{Dq})$.
\end{lthm}

If one were to work with a large number of auxiliary elliptic curves, one should get higher density sets $\sP$ that get arbitrarily close to $1$ with the number of curves, at the expense of lower density sets $\sQ$ and $\sQ_D$ -- see Remark \ref{Remark:future} for a brief discussion and Table \ref{densities} for the densities of $\sP$.
It would be very interesting to find such curves, for which significantly more computing power is certainly needed.
It would be interesting to solve the following problem.

\begin{problem}\label{problem}
Find a method that generates infinitely many such auxiliary curves.
\end{problem}

An affirmative answer to this problem would imply that the analogue of $\sP$ would have density $1$, while our current method of building $\sQ$ could potentially give a set of density $0$.
\begin{problem}\label{further-problem}
If the answer to Problem~\ref{problem} is affirmative, would it be possible to ensure that the resulting $\sQ$ is non-empty? Or even have positive density?
\end{problem}
See Remark~\ref{Remark:future} for further discussion related to Problems~\ref{problem} and \ref{further-problem}.

Finally, we fix a congruent number elliptic curve as our choice of auxiliary elliptic curve and adopt the techniques developed in \cite{GFP} to prove the following result.

\begin{lthm}[{Theorem~\ref{thm:cong}}]\label{thmD}
There is an explicit Chebotarev set of primes $\sP_\con$ with density $\frac{11}{16}$ such that \ref{H10} is unsolvable for the ring of integers of $L=\Q(\sqrt[3]{p},\sqrt{q})$ whenever $q$ is a  congruent number.
\end{lthm}

The congruent number problem and Goldfeld's conjecture predict that the set of primes $q$ in the statement of Theorem~\ref{thmD} should be $\frac{1}{2}$.
The recent tremendous progress on these profound problems made by D.~Kriz \cite{kriz} and A.~Smith \cite{smith} sheds important new light on this set.
\subsection*{Organization}
Including the introduction, there are seven sections and three appendices.
In the preliminary Section~\ref{section: prelim}, we record facts about \ref{H10} and Selmer groups.
In Section~\ref{section: strategy} we outline the strategy of using elliptic curves for proving that \ref{H10} is unsolvable in rings of integers of certain number fields.
In Section~\ref{section: improving GFP with one curve and two fields} we consider the setting of García-Fritz--Pasten's Theorem~\ref{thm:GFP} and provide improvements on the proportions.

In Sections~\ref{set Q one curves two K negative disc} and \ref{section: pair of ec}, we deviate from the setting of \cite{GFP}.
We work with auxiliary elliptic curves with negative minimal discriminant, allowing us to prove the insolubility in rings of integers of number fields disjoint from those considered in \cite{GFP}.
In the final Section~\ref{section: Congruent Number Curves}, we approach the problem with congruent number elliptic curves.
This curve was not covered by the methods of \cite{GFP} because of difficult behaviour at the prime $3$.
More precisely, it has \emph{supersingular} reduction at the prime $3$ with non-surjective image of the mod $3$ residual representation.

The appendices contain a MAGMA code provided to us by Harris B.~Daniels and the SAGE codes which were used for verifying various hypotheses in this article.

\section*{acknowledgements}
This project was initiated as a part of the IAS Summer Collaboration Program (2022).
We thank the Institute for Advanced Study for their hospitality and financial support.
We thank Harris B.~Daniels for sharing with us his code on computing joint images of Galois representations attached to two elliptic curves (see Appendix~\ref{app:Harris}).
We also thank John Cremona, Henri Darmon, Tim Dokchitser, Natalia García-Fritz, Daniel Kriz, Marc Masdeu, Kumar Murty, Robert Pollack, Karl Rubin, and Chris Skinner for answering many of our questions and their helpful comments during the preparation of this article.
DK is supported by a PIMS Postdoctoral Fellowship.
AL is supported by the NSERC Discovery Grants Program RGPIN-2020-04259 and RGPAS-2020-00096.
FS is supported by NSF grant 2001280 and Simons Grant 635320.

\section{Preliminaries}
\label{section: prelim}

\subsection{Facts related to \texorpdfstring{\ref{H10}}{}}
We record definitions and properties required for our purposes.
For further details we refer the reader to \cite[Chapters~5 and 7]{murty2019hilbert}.

\begin{definition}
\label{defn: Diophantine over a ring}
Let $R$ be a commutative unitary ring and let $n$ be a given positive integer. We say a set $S\subseteq R^n$ is \emph{Diophantine over $R$} if there exists a positive integer $m$ and a polynomial $P(x_1,\ldots,x_n,y_1,\ldots,y_m)$ with coefficients in $R$ such that $(a_1, \ldots, a_n)$ is in $S$ if and only if there exist elements $b_1,\ldots,b_m$ of $R$ for which $P(a_1,\ldots,a_n,b_1,\ldots,b_m) = 0.$ That is,
\[
(a_1, \ldots, a_n)\in S \Longleftrightarrow (\exists b_1, \ldots, b_m)(P(a_1, \ldots, a_n, b_1, \ldots, b_m)) =0 \quad
\]
When $n=1$, the set $S$ is said to be \emph{Diophantine in $R$}.
\end{definition}

\begin{example}\leavevmode
\begin{enumerate}
\item The set of non-negative integers $a\in \mathbb{Z}_{\geq 0}$ is Diophantine in $\mathbb{Z}$:
\[
a\in \mathbb{N} \Longleftrightarrow (\exists b)(a-b = 0)\Longleftrightarrow (\exists b_1,b_2,b_3,b_4)(a-b_1^2-b_2^2-b_3^2-b_4^2 = 0).
\]
\item Any finite set $S= \{a_1, \ldots, a_r\} $ is Diophantine: \[
a\in S \Longleftrightarrow (a-a_1)\ldots (a-a_r)=0.
\]
\item The set of composite numbers is Diophantine:
\[
a \text{ is a composite number } \Longleftrightarrow f(a, y_1, y_2) = a- (y_1 + 2)(y_2 + 2).
\]
\item If $S_1$ and $S_2$ are Diophantine, then $S_1 \cup S_2$ and $S_1\cap S_2$ are also Diophantine.
\end{enumerate}
\end{example}

\begin{definition}
\label{def: Integrally Diophantine}
Let $L_2/L_1$ be an extension of number fields.
If $\cO_{L_1}$ is Diophantine in $\cO_{L_2}$, then $L_2/L_1$ is said to be \emph{integrally Diophantine}.
\end{definition}

The following result tells us about the Diophantine relationships between algebraic number fields.
\begin{labelledconj}{Transitive Property}{Thm:transitivity}
If $L_3/L_2/L_1$ is a tower of number fields and both $L_2/L_1$
and $L_3/L_2$ are integrally Diophantine, then so is $L_3/L_1$.
\end{labelledconj}

\begin{proof}
See \cite[Theorem~2.1]{SS89}.
\end{proof}

\begin{labelledconj}{Shlapentokh's Theorem}{Thm:Shlapentokh}
Let $L_2/L_1$ be an extension of number fields.
Suppose that there is an elliptic curve $E/L_1$ such that $\rank E(L_2)=\rank E(L_1)>0$.
Then $L_2/L_1$ is integrally Diophantine.
\end{labelledconj}

\begin{proof}
See \cite{Shlapen}.
\end{proof}


\subsection{Facts about Selmer groups of elliptic curves}
\label{S: Selmer groups}
Throughout this article, $\ell$ is an odd prime number and $E$ is an elliptic curve defined over $\Q$ with good reduction at $\ell$.

Given a module $A$ over the absolute Galois group $G_{L}$, and $i\geq 0$, the cohomology group $H^i(L, A)$ is defined to be the discrete cohomology group $H^i(G_{L}, A)$.

The \emph{$\ell$-Selmer group} of $E$ over $L$ is defined as the kernel of the following restriction map
\[
\Sel_{\ell}(E/L):=\ker\left( H^1\left(L,E[\ell]\right)\longrightarrow \prod_{w} H^1\left(L_w, E(\overline{L_w}) \right)[\ell]\right).
\]
An equivalent definition of the $\ell$-Selmer group given in \cite[\S~1 Corollary~6.6]{Milne} is the following: let $S$ be a finite set of primes $L$ containing all the archimedean primes, the primes above $\ell$, and the primes where $E$ has bad reduction.
Then
\[
\Sel_{\ell}(E/L):=\ker\left( H^1\left(L_S/L,E[\ell]\right)\longrightarrow \prod_{w\in S} H^1\left(L_w, E(\overline{L_w}) \right)[\ell]\right).
\]
Here, $L_S/L$ is the maximal extension of $L$ unramified outside the set $S$.

Now we record a result of J.~Brau crucial for our results.
Let $r$ be a prime of semi-stable reduction and $r^s \pdiv a$, i.e., $a = r^s d$ such that $\gcd(d,r)=1$.
Since $r$ is semi-stable, we can write $j = r^{-n} \frac{a}{b}$ (where $a,b$ are coprime to $r$).
Define
\[
u_{r,a} = d^n \left(\frac{a}{b}\right)^s.
\]

\begin{theorem}
\label{thm: brau}
Let $\ell$ be an odd prime and $E/\Q$ be an elliptic curve with good reduction at $\ell$.
Let $k=\Q(\mu_{\ell})$ and suppose that $\Sel_{\ell}(E/k)$ is trivial.
Consider the family of $\ell$-extensions $k_{a} = k(\sqrt[\ell]{a})$.
Write $S$ to denote the set of primes containing the primes above $\ell$, the primes of bad reduction of $E$, the primes which ramify in $k_a/k$, and the archimedean primes.
Denote by $G_v$ the Galois group $\Gal(k_{a,w}/k_v)$ (where $w$ is a prime of $k_a$ lying above $v$) and set $\delta_v = \dim_{\F_\ell} H^1\left(G_v, E(k_{a,w}) \right)$.
Denote by $c_v$ is the Tamagawa number of $E/k_v$ and by $q_v$ the size of the residue field of $k_v$.
Then
\[
\dim_{\F_{\ell}}\Sel_{\ell}\left( E/k_a\right)^{\Gal(k_a/k)} = \sum_{v\in S} \delta_v.
\]
The values of $\delta_v$ (when $v$ is not a prime of additive reduction) are given in the following table:

\noindent
\begin{tabular}{@{\vrule width 1pt\ }l|c|c|c@{\ \vrule width 1pt}}
\noalign{\hrule height 1.2pt}
Reduction type of $E$ at $v$ & $v$ ramified in $k_{a}/k$ & $v$ inert in $k_{a}/k$ & $v$ split in $k_{a}/k$ 
\\
\noalign{\hrule height 1.2pt}
good ordinary, $v \mid \ell$ & $
\delta_v=
\begin{cases}
1\text{ or } 2 & \text{if $\tilde{E}(\kappa_v)[\ell]\neq 0$} \\
0 & \text{otherwise}
\end{cases}
$ & \multicolumn{2}{@{\ }c@{\ \vrule width 1.2pt}}{
$\delta_v = 0$
} \cr
\hline
good supersingular, $v \mid \ell$ & $
\delta_v =
\begin{cases}
\ell-2 & \text{if $\ell\mid a$} \\
0 & \text{otherwise}
\end{cases}
$ & \multicolumn{2}{@{\ }c@{\ \vrule width 1.2pt}}{
$\delta_v = 0$
} \cr
\hline
good, $v \nmid \ell$ & $\delta_v = \dim_{\mathbb{F}_\ell} 	\tilde{E}(\kappa_v)[\ell]$ & \multicolumn{2}{@{\ }c@{\ \vrule width 1.2pt}}{
$\delta_v = 0$
} \cr
\hline
split multiplicative & $\delta_v = \begin{cases}
1 & \text{if $u_{r,a}^{\frac{q_v-1}{\ell}} \equiv 1 \mod r$}\\
0 & \text{otherwise}
\end{cases}
$
& $\delta_v = \begin{cases}
1 & \text{if $\ell \mid c_v$} \\
0 & \text{otherwise}
\end{cases}
$ & $\delta_v = 0$ \cr
\hline
non-split multiplicative & \multicolumn{3}{@{\ }c@{\ \vrule width 1.2pt}}{
$\delta_v = 0$
}\cr
\hline
\end{tabular}
\end{theorem}

\begin{proof}
The first assertion of the theorem is \cite[Proposition 5.2]{Bra14}.
We note that the description of the set $S$ is not clearly mentioned in \emph{op. cit.}: a partial description of the set $S$ is included in Proposition~2.1 of \emph{op. cit.}
We further need to include the primes which ramify in $k_a/k$ to ensure that $k_a$ is a finite extension of $k$ contained in $k_S$.
The values of $\delta_v$ at non-additive primes are given in the statement of Theorem~1.1 of \emph{op. cit.}
\end{proof}

\section{Strategy}
\label{section: strategy}

In this section, we outline the strategy of the proofs of Theorems~\ref{thmA}-\ref{thmD}.
Our approach is inspired by the techniques developed in \cite{GFP} but is more direct and free of Iwasawa theoretic arguments.
The following consequence of \ref{Thm:Shlapentokh} plays a key role.
\begin{proposition}
\label{prop: GFP 3.3}
Let $F/\Q$ be any number field and $K/\Q$ be a quadratic extension.
Consider the compositum $L=F\cdot K$.
If there is an elliptic curve $E/\Q$ satisfying 
\begin{enumerate}[\textup{(}i\textup{)}]
\item $\rank E(F)=0$ and
\item $\rank E(K)>0$,
\end{enumerate}
then $L/F$ is integrally Diophantine, i.e., $\cO_F$ is Diophantine in $\cO_L$.
\end{proposition}

\begin{proof}
See \cite[Proposition~3.3]{GFP}.
\end{proof}

If we further know that $F/\Q$ is integrally Diophantine (i.e., $\Z$ is Diophantine in $\cO_F$), then by the \ref{Thm:transitivity}, $L/\Q$ is also integrally Diophantine.
The following argument is standard, but we repeat it for the reader's convenience.

\medskip
\noindent \emph{Claim:} If $\Z$ is Diophantine in $\cO_L$ (Definition~\ref{defn: Diophantine over a ring}), then the analogue of \ref{H10} for $\cO_L$ has a negative solution.

\medskip
\noindent \emph{Justification:}
Let $P$ be the polynomial in $\cO_{L}[X_i, Y_{j,j'}]_{1\leq i,j\leq n, 1\leq j'\leq l}$
that shows that $\Z$ is Diophantine in $\cO_L$ and let $f_1, \ldots, f_m\in \Z[X_1, \ldots, X_n]$.
By Definition~\ref{defn: Diophantine over a ring}, the collection of polynomials 
\[
f_1, \ldots, f_m, P(X_1, Y_{1,1}, \ldots, Y_{1,k}), \ldots, P(X_n, Y_{n,1}, \ldots, Y_{n,k})\]
is solvable in $\cO_L$ precisely if the polynomials $f_1, \ldots, f_m$ are solvable in $\Z$.
In other words, if we are able to decide the solubility of polynomials over $\cO_L$, then we would be able to decide the solubility of $f_1, \ldots, f_m$ over $\Z$.
This is a contradiction to \ref{Thm:Matiyasevich}, completing the justification.

Thus, the tasks at hand are the following:
\begin{enumerate}
\item Find a rank 0 elliptic curve $E/\Q$ and two families of number fields:
\begin{itemize}
\item one family such that rank of $E/\Q$ does not jump in the number fields.
\item another family of (quadratic) extensions such that rank does jump.
\end{itemize}
\item Determine how big these families are.
\end{enumerate}

\subsection{Step 1: Rank Stabilization in cubic extensions}

Throughout this section, $E/\Q$ is an elliptic curve with good reduction at a fixed prime $\ell$.
Consider the Galois group $G_{E,\ell}:=\Gal(\Q(E[\ell])/\Q)$, and note that $G_{E,\ell}$ may be viewed as a subgroup of ${\GL}_2(\Z/\ell\Z)$ via the residual representation 
\[
\bar{\rho}_{E,\ell}:G_{E,\ell}\hookrightarrow {\GL}_2(\Z/\ell\Z).
\]
Assume that $p$ is a prime coprime to the conductor of $E$, i.e., $p$ is unramified in $\Q(E[\ell])$.
Let $\sigma_p\in G_{E,\ell}$ be the Frobenius at $p$.
The trace and determinant of $\bar{\rho}(\sigma_p)$ are as follows
\begin{align*}
\trace\ \bar{\rho}_{E,\ell}(\sigma_p)&=a_p(E)=p+1-\# \widetilde{E}(\F_p),\\
\det\bar{\rho}_{E,\ell}(\sigma_p)&=p.
\end{align*}
For a prime $v|p$ of $\QQ(\mu_\ell)$, let $f$ be the integer such that the residue field of $\QQ(\mu_\ell)$ at $v$ is given by $\kappa_v=\F_{p^f}$ (note that $f$ is independent of the choice of $v$).
According to a formula of A.~Weil (see \cite[Theorem~V.2.3.1]{silverman2009}),
\[
 \#E(\kappa_v)=p^f+1-\alpha^f-\beta^f\equiv 2-\alpha^f-\beta^f\mod{\ell},
\]
where $\alpha$ and $\beta$ are the eigenvalues of $\bar{\rho}_{E,\ell}( \sigma_p)$.

\begin{definition}\label{def:HEl}
For $g\in G_{E,\ell}$, denote by $f(g)$ the smallest positive integer $f$ satisfying 
\begin{equation}
\label{eqn: det condition}
{\det}\bar{\rho}_{E,\ell}(g)^f=1.
\end{equation}
Define the set $H_{E,\ell}$ to consist of all $g\in G_{E,\ell}$ such that the eigenvalues $\alpha,\beta\in \overline{\F}_\ell$ of $\bar{\rho}(g)$ satisfy
\begin{equation}
\label{eqn: condn for HEl}
\alpha^{f(g)}+\beta^{f(g)}\neq 2.
\end{equation}
\end{definition}

Note that \eqref{eqn: det condition} implies that $(\alpha\beta)^{f(g)}=1$.
Therefore, the condition \eqref{eqn: condn for HEl} is equivalent to $\alpha^{f(g)}\neq 1$.

\begin{definition}
\label{defi: Chebotarev set}
A set of primes $\mathcal{S}$ is called a \emph{Chebotarev set} if there
is a Galois extension $K/\Q$ and a conjugacy-stable set $\cC \subseteq \Gal(K/\Q)$ such that up to a finite set, $\mathcal{S}$ agrees with
the set $\{p\colon \Frob_p \in \cC\}$.

Finite unions, finite intersections, and complements of Chebotarev sets are again Chebotarev.
\end{definition}

The Chebotarev density theorem states that if $\mathcal{S}$ arises from $K$ and $\cC$ as above, then the \emph{density of $\mathcal{S}$} defined as the limit
\[
\lim_{x\rightarrow\infty} \frac{\# \mathcal{S} \cap [1,x]}{\pi(x)}
\]
exists and is equal to $\#\cC/[K : \Q]$.
Here, we have used the standard notation $\pi(x)$ to denote the number of prime numbers $\leq x$.

\begin{lemma}
\label{Hep over Gep lemma}
For a prime $p\neq\ell$, let $v$ be a prime of $\Q(\mu_\ell)$ above $p$, and $\kappa_v$ be the residue field at $v$.
The 
density of primes $p$ coprime to the conductor of $E$ so that $E(\kappa_v)[\ell]=0$ is $\left(\frac{\# H_{E,\ell}}{\# G_{E,\ell}}\right)$.
\end{lemma}

\begin{proof}
This is \cite[Lemma~8.9]{KLR22} -- we repeat the brief proof for the convenience of the reader as we will be using it later.
By definition, $\sigma_p\in H_{E,\ell}$ if and only if $E(\kappa_v)[\ell]=0$.
The result follows from the Chebotarev density theorem.
\end{proof}

We now work towards finding a set of primes $\sP$ such that for all $p\in \sP$ the Mordell--Weil rank of $E\left(\Q(\sqrt[3]{p})\right) $ is zero.
We prove a modified version of \cite[Theorem~4.1]{GFP}.

\begin{theorem}
\label{thm:growth-rank}
Let $\ell>2$ be a prime.
Let $E/\Q$ be an elliptic curve with conductor $N$ such that
\begin{enumerate}[\textup{(}1\textup{)}]
\item $E$ has good reduction at $\ell$;
\item $\Sel_{\ell}\left(E/\Q(\mu_{\ell})\right)=0$
\item $\ell\nmid \Tam\left(E/\QQ(\mu_\ell)\right)\cdot \#\tilde E_\ell(\F_\ell)$.
\end{enumerate}
Consider the set of prime numbers
\[
\sP(E,\ell)=\{p:p\nmid N, a_v(E)\not\equiv 2\mod\ell\},
\]
where $v$ denotes any prime of $\QQ(\mu_\ell)$ lying above $p$ and $a_v(E)=1+\# \kappa_v-\#\tilde E_v(\kappa_v)$.
Then for every $\ell$-power free integer $a>1$ supported on $\sP(E,\ell)$, the Selmer group $\Sel_{\ell^\infty}\left(E/\QQ(\mu_\ell,\sqrt[\ell]a)\right)$ is trivial.
\end{theorem}

\begin{proof}
We write $k=\Q(\mu_\ell)$ and $k_a = k(\sqrt[\ell]{a})$.
Note that $k_a/k$ is a Galois extension of number fields of degree $\ell$.
Let $G$ denote its Galois group.

We make the following crucial observation:
Suppose $v\nmid \ell$ is a prime in $k$ that ramifies in $k_a$.
Then we know that $v$ divides $a$.
But $a$ is in the support of $\sP(E,\ell)$. This forces $v$ to lie above a prime $p$ for some $p\in \sP(E,\ell)$, justifying our choice of letter $v$. We thus have by assumption that 
\[
a_v(E)\not\equiv 2 \mod{\ell}.
\]
From the proof of Lemma~\ref{Hep over Gep lemma}, the primes of good reduction $v$ satisfying the above condition are those which satisfy
\[\widetilde{E}_v(\kappa_v)[\ell]=(0).\]
Further, Lemma~\ref{Hep over Gep lemma} itself tells us that $\sP(E,\ell)$ is a Chebotarev set with density $\left(\frac{\# H_{E,\ell}}{\# G_{E,\ell}}\right)$.

In view of the assumption that $\Sel_{\ell}\left(E/\Q(\mu_\ell)\right)=0$, we may apply Theorem~\ref{thm: brau}.
We shall consider two separate cases, namely when $E$ is semi-stable and when $E$ is not semi-stable.

Let us first consider the semi-stable case.
The only primes that can ramify are the primes dividing $\ell a$.
Since $a$ lies in the support of $\sP(E,\ell)$, it is clear that when $v\mid a$, we have $\delta_v =0$, where $\delta_v$ is given as in the statement of Theorem~\ref{thm: brau} (we are in the row of $v\nmid\ell, v$ good).
Once again from Theorem~\ref{thm: brau} we see that the hypotheses on the Tamagawa number and the fact that $\ell$ is non-anomalous imply that $\Sel_{\ell}\left(E/k_a\right)^G=0$ when $E/\Q$ has no prime of additive reduction.
By \cite[Proposition~1.6.12]{NSW08}, we deduce that $\Sel_{\ell}\left(E/k_a)\right)=0$.

Suppose now that $E$ is not semi-stable.
Again by \cite[Proposition~1.6.12]{NSW08}, to show that $\Sel_{\ell}\left(E/k_a\right)=0$, it suffices to show that $\Sel_{\ell}\left(E/k_a\right)^{G}=0$.

As discussed in \cite[Remark~5.3]{Bra14}, there is an isomorphism of $\F_\ell$-vector spaces
\[
\Sel_{\ell}\left(E/k_a\right)^{G} \simeq \bigoplus_{v\in S} H^1\left(G_v, E(k_{a,w}) \right)[\ell],
\]
where $w$ is a prime of $k_a$ lying above $v$.
Just as before, Theorem~\ref{thm: brau} tells us that when $v$ is a prime of good reduction or a prime of multiplicative reduction, the summand $H^1\left(G_v, E(k_{a,w}\right)[\ell]$ vanishes.
It remains to study the summands arising from the additive primes.
Let $v$ be a prime of additive reduction and write $r$ for the rational prime lying below $v$.
Recall that $E$ has good reduction at $\ell$.
It tells us that $\ell \neq r$, which implies the following isomorphism
\[
H^1\left(G_v, E(k_{a,w}) \right)[\ell] \simeq H^1\left(G_v, E(k_{a,w})[\ell] \right).
\]
Therefore, it is enough to show that
\begin{equation}
\label{triviality}
E(k_{a,w})[\ell]=0.
\end{equation}

We have assumed that $\ell\nmid \Tam(E/k)$; it follows $\ell\nmid \Tam(E/k_a)$.
Since $v$ is a prime of additive reduction reduction and $r\nmid \ell a$, the base change $E/k_a$ still has additive at $w$.
It follows from \cite[Theorem~2.30]{Was08} that
\[
\abs{\widetilde{E}^{\ns}(\kappa_w)} = \abs{\kappa_w} \not\equiv 0\pmod{\ell}.
\]
Recall that (see \cite[Chapter VII, Proposition~2.1]{silverman2009})
\[
\abs{ E(k_{a,w})[\ell]} = \abs{ E(k_{a,w})/\ell E(k_{a,w}) }= \abs{\widetilde{E}^{\ns}(\kappa_w)[\ell]} \times \abs{c_w^{(\ell)}(E/k_a)} =1,
\]
proving \eqref{triviality}.
Therefore, we conclude that
\[
\Sel_\ell\left(E/k_a\right)^G=\Sel_{\ell}\left(E/k_a\right) =0.
\]
\end{proof}

\begin{remark}
We compare our result with \cite[Theorem~4.1]{GFP}.
There are three main differences:
\begin{itemize}
\item we can define a larger set $\sP(E,\ell)$, i.e., without assuming that $p\equiv 1 \pmod{\ell}$.
\item we only need to assume that $\ell$ is a prime of good reduction (not good ordinary reduction).
\item we do not require that the mod $\ell$ representation is surjective.
\end{itemize}
The authors in \cite{GFP} only assume that $\rank E(k)=0$ and that $\Sha(E/k)[\ell]=0$.
However, we have the additional assumption that $E/k$ has no $\ell$-torsion.
The advantage is that we obtain that the $\ell$-Selmer group over $k_a$ is \emph{trivial}, not just finite.
But we point out that the auxiliary elliptic curve that is used in \cite{GFP} also satisfies this extra hypothesis that we impose.

The same Iwasawa theoretic argument used in \cite[Theorem~4.1]{GFP} can be repeated if $\ell$ is a prime of good ordinary reduction, the set $\sP(E,\ell)$ is as in our current theorem, and with no hypothesis on the mod $\ell$ representation being surjective.
\end{remark}

\subsection{Step 2: Rank jump in a quadratic extension}
We will rely heavily on the work of D.~Kriz and C.~Li \cite{krizli} to construct a set of primes $\sQ$ such that for all $q\in \sQ$, the Mordell--Weil rank of $E\left(\Q(\sqrt{-q})\right)$ \emph{or} $E\left(\Q(\sqrt{\abs{d_K}q})\right)$ is $1$.

\subsubsection{Recollections from \texorpdfstring{\cite{krizli}}{}}
We remind the reader of two results from \cite{krizli} which will allow us to obtain information on the rank jump upon performing a base-change.

Let $E/\Q$ be an elliptic curve of conductor $N$ and $K$ be a quadratic field.
Define the set of primes
\[
\sQ_{K}(E) = \left\{q\colon q\nmid 2N, \ q \text{ splits in }K, \ \text{ and } a_q(E)\equiv 1\pmod{2} \right\}.
\]

\begin{theorem}
\label{KL-result}
Let $E/\Q$ be an elliptic curve of conductor $N$ with trivial $2$-torsion.
Let $K$ be an imaginary quadratic field satisfying the Heegner hypothesis for $N$.
Let $P$ denote a Heegner point 
Suppose that
\begin{equation}
\label{star}\tag{$\star$}
2 \text{ splits in } K \text{ and } \ \frac{\abs{\widetilde{E}^{\ns}(\F_2)} \cdot \log_{\omega_E}(P)}{2} \not\equiv 0 \pmod{2}.
\end{equation}
Then for each square-free integer $d\equiv 1 \pmod{4}$ supported on $\sQ_K(E)$, the rank part of the Birch--Swinnerton-Dyer conjecture is true for $E^{(d)}/\Q$ and $E^{(d\cdot d_K)}/\Q$.
One of these two curves has (algebraic and analytic) rank zero and the other has (algebraic and analytic) rank one.

If further, the Tamagawa number at $2$ is odd,
\[
\rank E(\Q) = \rank E^{(d)}(\Q) \text{ if and only if } \Delta_E <0 \text{ or }, \Delta_E>0 \text{ and } d>0,
\]
where $\Delta_E$ is the minimal discriminant of $E$.
\end{theorem}

\begin{proof}
See \cite[Theorem~4.3 and Corollary~5.11]{krizli}.
\end{proof}

The next two results are restatements of the above theorem, but we record them separately since they will be required for our calculations.
\begin{corollary}
\label{to be use for the real quadratic version}
Let $K$ be an imaginary quadratic field and let $E/\Q$ be an elliptic curve satisfying the following properties:
\begin{enumerate}[\textup{(}i\textup{)}]
\item $\rank E(\Q)=0$.
\item $E(\Q)[2]$ is trivial.
\item $(E,K)$ satisfies the Heegner Hypothesis.
\item Hypothesis~\eqref{star} holds.
\item $c_2(E)$ is odd.
\item $\Delta_E <0$.
\end{enumerate}
Then, $\rank E^{(d\cdot d_K)} =1$ for all $d\equiv 1\pmod{4}$ supported on $\sQ_{K}(E)$.
\end{corollary}

\begin{proof}
Since we are assuming that $\Delta_E<0$, this means that for all $d\equiv 1\mod{4}$,
\[
\rank E(\Q) = \rank E^{(d)}(\Q).
\]
The first assertion of Theorem~\ref{KL-result} implies that $E^{(d\cdot d_K)}/\Q$ has (analytic and algebraic) rank one.
\end{proof}

In exactly the same way, we can prove the following result.
\begin{corollary}
\label{basically what GFP used}
Let $K$ be an imaginary quadratic field and let $E/\Q$ be an elliptic curve satisfying the following properties:
\begin{enumerate}[\textup{(}i\textup{)}]
\item $\rank E(\Q)=0$.
\item $E(\Q)[2]$ is trivial.
\item $(E,K)$ satisfies the Heegner Hypothesis.
\item Hypothesis~\eqref{star} holds.
\item $c_2(E)$ is odd.
\item $\Delta_E >0$ and $d<0$.
\end{enumerate}
Then, $\rank E^{(d)} =1$ for all $d\equiv 1\pmod{4}$ supported on $\sQ_{K}(E)$.
\end{corollary}

\subsection{Step 3: Determining the size of these families of extensions}

\subsubsection{The density of $\sP$}\label{S:densityP}
We consider the special case where $\overline\rho_{E,\ell}$ is surjective.
We shall consider more general cases in subsequent sections.
We begin with the following elementary lemma:
\begin{lemma}
\label{lem:proportion}
Suppose that $\bar\rho_{E,\ell}$ is surjective.
When $\ell=3$, we have
\[
\frac{\# H_{E,\ell}}{\# G_{E,\ell}}=\frac{9}{16}.
\]
\end{lemma}

\begin{proof}
The group $G_{E,\ell}\cong\GL_2(\Z/\ell\Z)$ has $\ell(\ell-1)^2(\ell+1)=48$ elements.
As in \cite[Appendix~A]{KLR22}, we may divide $G_{E,\ell}$ into the following conjugacy classes:
\begin{itemize}
 \item Let $C_{a,b}$ be the set of diagonalizable matrices with eigenvalues $a,b\in {\F_\ell}^\times$ with $a\ne b$.
 We have $(\ell-1)(\ell-2)/2$ choices of $C_{a,b}$ and for each choice, $\#C_{a,b}=\ell(\ell+1)$.
 \item Let $\cC_a$ be the set of non-diagonal matrices with one single eigenvalue $a\in {\F_\ell}^\times$.
 There are $(\ell-1)$ choices for $\cC_a$ and for each choice, $\#\cC_a=\ell^2-1$.
 \item Let $D_a=\left\{\begin{pmatrix}a&0\\0&a\end{pmatrix}\right\}$, $a\in \F_\ell^\times$.
 Then, there are $(\ell-1)$ choices for $a$ and for each choice $\#D_a=1$.
 \item Let $\cE_{\lambda}$ be the set of matrices whose eigenvalues are $\lambda$ and $\lambda'$, where $\lambda\in\mathbb{F}_{\ell^2}\setminus\F_\ell$ and $\lambda'$ is the conjugate of $\lambda$.
 There are $\ell(\ell-1)/2$ choices for $\lambda$ and for each choice of $\lambda$, $\#\cE_\lambda=\ell^2-\ell$.
\end{itemize}

All elements in $\cE_\lambda$ belong to $H_{E,\ell}$, giving $\ell^2(\ell-1)^2/2=18$ elements.
For contributions from $D_a$ and $\cC_a$, we seek $a\in \F_\ell^\times$ whose order is even.
The only choice is $a=2$.
This gives $(\ell^2-1)+1=9$ elements.
Finally, there is no element in $C_{a,b}$ belonging to $H_{E,\ell}$, since if $g\in C_{a,b}$, the eigenvalues of $g$ are $1$ and $2$, forcing $f(g)=2$ and $a^{f(g)}+b^{f(g)}=2$.
Summing up, the result follows.
\end{proof}

\begin{remark}
\cite[Proposition~4.6]{GFP} when applied with $\ell=3$, yields a proportion equal to $5/16$.
The proportion obtained in Lemma~\ref{lem:proportion} exceeds this by $1/4$.
\end{remark}

\begin{corollary}\label{cor:stable-rank-1-curve}
Suppose that $E/\QQ$ is an elliptic curve satisfying the hypotheses of Theorem~\ref{thm:growth-rank} with $\ell=3$.
Then, there exists a Chebotarev set of primes $\sP(E,3)$ with density $\frac{9}{16}$ such that for all $p\in \sP(E,3)$, the Mordell--Weil rank of $E\left(\QQ(\sqrt[3]{p})\right)=0$.
\end{corollary}
\begin{proof}
It follows from Theorem~\ref{thm:growth-rank} (with $\ell=3$), Lemma~\ref{lem:proportion} and the Chebotarev density theorem.
\end{proof}

\begin{remark}
The main improvement in our result compared to \cite[Theorem~4.1]{GFP} is that we work with a larger set of primes $\sP(E,\ell)$.
In particular, we do not need to impose the hypothesis that the primes $p\in \sP(E,\ell)$ satisfy the additional hypothesis that $p \equiv 1 \pmod{\ell}$.
\end{remark}

\subsubsection{The density of $\sQ$}

To determine the density of the set $\sQ$ we use the following lemma.

\begin{lemma}
\label{GFP 5.1}
Let $E/\Q$ be an elliptic curve of conductor $N$ and let $K$ be a fixed quadratic field.
Define
\begin{align*}
\sQ(E,K) &= \left\{ q\colon q\nmid 2N, q \text{ splits in } K, \ a_q(E)\equiv 1\pmod{2}\right\}\\
\sQ^+(E,K) &= \left\{ q\colon q\in \sQ(E,K) \text{ and } q\equiv 1\pmod{4}\right\}\\
\sQ^-(E,K) &= \left\{ q\colon q\in \sQ(E,K) \text{ and } q\equiv -1\pmod{4}\right\}
\end{align*}
If the mod $2$ representation is surjective and $\Gal\left(\Q(E[2])\right)/\Q$ does not contain $K$ or $\Q(\sqrt{-1})$, then the density of the sets $\sQ(E,K)$, $\sQ^+(E,K)$, and $\sQ^-(E,K)$ are $\frac{1}{6}$, $\frac{1}{12}$, and $\frac{1}{12}$ respectively.
\end{lemma}

\begin{proof}
See \cite[Lemma~5.1]{GFP}.
\end{proof}

\section{\texorpdfstring{\ref{H10}}{} for rings of integers of \texorpdfstring{$\Q(\sqrt[3]{p},\sqrt{-q})$}{}}
\label{section: improving GFP with one curve and two fields}

The main goal of this section is to prove Theorem~\ref{thmA}, which provides an improvement of \cite[Theorem~1.2]{GFP}.
Our calculations in \S\ref{S:densityP} already allows us to obtain a set $\sP$ with larger density.
We shall work with two auxiliary imaginary quadratic fields,  allowing us to improve \cite[Lemma~6.4]{GFP} by enlarging the set $\sQ$.
In other words,  we are able to larger sets $\sP$ and $\sQ$ (compared with  Theorem~\ref{thm:GFP}) such that for all $p\in \sP$ and $q\in \sQ$, \ref{H10} has a negative solution for rings of integers of $\Q(\sqrt[3]{p},\sqrt{-q})$.

\subsection{Rank jump in multiple quadratic fields}
Fix an elliptic curve $E/\QQ$.
We work with more than one auxiliary imaginary quadratic fields and study rank jumps  of $E$ in any one of these fields.

\begin{lemma}
\label{proportion of set Q for two imaginary quadratic}
Let $K_{(1)}, \ldots, K_{(n)}$ be $n$ distinct imaginary quadratic fields.
Suppose that $E/\QQ$ is an elliptic curve such that its mod $2$ representation is surjective and that the Galois extension $\Q\left(E[2]\right)/\Q$ does not contain $\Q(\sqrt{-1})$ nor $K_{(i)}$ for $1\leq i\leq n$.
Then, 
\[
\sQ_{(n)}(E) = \left\{q\colon q\equiv -1\pmod{4}, \ q \text{ splits in any one of }K_{(1)} \text{ or }\ldots \text{ or } K_{(n)}, \ a_q(E) \equiv 1\pmod{2} \right\}
\]
is a Chebotarev set of primes with density $\frac{1}{6}\times\left( 1-\frac{1}{2^n}\right)$.
\end{lemma}

\begin{proof}
We can rewrite
\begin{align*}
\sQ_{(n)}(E) = \sQ_{K_{(1)}}(E) \cup \sQ_{K_{(2)}}(E) \cup \ldots \cup \sQ_{K_{(n)}(E)}.
\end{align*}
Therefore, it is clear that $\sQ_{(n)}(E)$ is a Chebotarev set.
Using a matrix counting argument\footnote{Alternatively, one may use an inclusion-exclusion argument combined with the fact that the density of the set $\left\{q\colon q \text{ splits in }K_{(i)}, \ a_q(E) \equiv 1\pmod{2} \right\}$ is $\frac{1}{6}$ for each $i$.} as in \cite[Lemma~6.4]{GFP}, it is easy to see that the set
\[
\left\{q\colon q \text{ splits in any one of }K_{(1)} \text{ or }\ldots \text{ or } K_{(n)}, \ a_q(E) \equiv 1\pmod{2} \right\}
\]
has density $\left(1-\frac{1}{2^n}\right)\times \frac{1}{3}$.
Here, the first factor comes from the fact that we must avoid counting those primes $q$ which are inert in all of the imaginary quadratic fields $K_{(1)}, \ldots, K_{(n)}$.
By the Chebotarev density theorem, this happens for exactly $\frac{1}{2^n}$ proportion of the prime numbers.
Since we have assumed that $\Q(E[2])/\Q$ does not contain $\Q(\sqrt{-1})$ we can apply Lemma~\ref{GFP 5.1} to conclude that the desired density is
\[
\frac{1}{2}\times \left(1-\frac{1}{2^n}\right)\times \frac{1}{3}.
\]
\end{proof}

\begin{lemma}
\label{one curve rank 1 in two extensions}
Consider the elliptic curve $E=$\href{https://www.lmfdb.org/EllipticCurve/Q/557/b/1}{$557b1$}.
Then, there exists a Chebotarev set of primes $\sQ$ with density $\frac{7}{48}$ such that for all $q\in \sQ$,
\[
\rank E\left(\QQ(\sqrt{-q})\right)=1.
\]
\end{lemma}

\begin{proof}
One can further verify using \cite{Sage} that
\begin{enumerate}[\textup{(}i\textup{)}]
\item $E$ has Mordell--Weil rank 0 over $\Q$ and trivial $2$-torsion.
\item The minimal discriminant of $E$ is positive.
\item The Tamagawa number at $2$ is odd.
\item $E$ has surjective mod $2$ representation.
\item The only degree to subfield of the Galois extension $\Q\left(E[2]\right)/\Q$ has discriminant $557$.
\end{enumerate}
One can check that the Heegner Hypothesis is satisfied for the pairs $(E,\Q(\sqrt{-7}))$, $(E,\Q(\sqrt{-79}))$ and $(E,\Q(\sqrt{-127}))$,
and that Hypothesis~\eqref{star} holds for each of these imaginary quadratic fields.

Choose $\sQ = \sQ_{\Q(\sqrt{-7})}(E) \cup \sQ_{\Q(\sqrt{-79})}(E) \cup \sQ_{\Q(\sqrt{-127})}(E)$.
We know from Lemma~\ref{proportion of set Q for two imaginary quadratic} that $\sQ$ is a Chebotarev set of density $\frac{7}{48}$.
Furthermore, Corollary~\ref{basically what GFP used} tells us that for integers $d=-q\equiv 1\pmod{4}$ that are supported on $\sQ_{\Q(\sqrt{-7})}(E)$ or $\sQ_{\Q(\sqrt{-79})}(E)$ or $\sQ_{\Q(\sqrt{-127})}(E)$, the Mordell--Weil rank of the twisted curve $E^{(d)}$ is $1$.
It follows that for each $q\in \sQ$,
\[
\rank E\left(\Q(\sqrt{-q})\right) = \rank E(\Q) + \rank E^{(-q)}(\Q) = 1.
\]
This completes the proof of the lemma.
\end{proof}

\subsection{Application to \texorpdfstring{\ref{H10}}{}.}
We now prove Theorem~\ref{thmA}.
Our result improves upon \cite[Theorem~1.2]{GFP} and provides a larger class of extensions where \ref{H10} is unsolvable.


\begin{theorem}
\label{thm: improve P and Q of GFP}
There are explicit Chebotarev sets of primes $\sP$ and $\sQ$ of density $\frac{9}{16}$ and $\frac{7}{48}$, respectively such that for all $p\in \sP$ and $q\in \sQ$, the analogue of \ref{H10} is unsolvable for the ring of integers of $L=\Q(\sqrt[3]{p},\sqrt{-q})$.
\end{theorem}

\begin{proof}
We will work with the rank $0$ auxiliary elliptic curve $E=\href{https://www.lmfdb.org/EllipticCurve/Q/557/b/1}{557b1}$.
This elliptic curve satisfies the hypotheses of Theorem~\ref{thm:growth-rank}, see \cite[Lemma~6.2]{GFP}\footnote{
One extra condition we need to check using \cite{Sage} is that $E(\Q(\mu_3))[3]$ is trivial.}.
We set  $\sP$ and $\sQ$ to be the sets of primes given in Corollary~\ref{cor:stable-rank-1-curve} and Lemma~\ref{one curve rank 1 in two extensions} respectively.
For each $p\in \sP$ and $q\in \sQ$, the rank of the elliptic curve remains of Mordell--Weil rank $0$ over the cubic extension $\Q(\sqrt[3]{p})$ and has rank $1$ over the imaginary quadratic extension $\Q(\sqrt{-q})$.
Proposition~\ref{prop: GFP 3.3} asserts that $\Q(\sqrt[3]{p}, \sqrt{-q})/\Q(\sqrt[3]{p})$ is integrally Diophantine.
But, it is well-known that $\Q(\sqrt[3]{p})/\Q$ is integrally Diophantine since $\Q(\sqrt[3]{p})$ has exactly one complex place.
The result follows from the \ref{Thm:transitivity}.
\end{proof}

\section{\texorpdfstring{\ref{H10}}{} for rings of integers of \texorpdfstring{$\Q(\sqrt[3]{p},\sqrt{Dq})$}{}}
\label{set Q one curves two K negative disc}

In this section, we work with a fixed elliptic curve of negative minimal discriminant and an imaginary quadratic field of discriminant $-D$.
We will show that there exist Chebotarev sets $\sP$ and $\sQ_D$ of positive density such that for all $p\in \sP$, and $q\in \sQ_D$, \ref{H10} has a negative solution for two families of rings of integers of $\Q(\sqrt[3]{p}, \sqrt{Dq})$.
Unlike in the previous section, the extensions we obtain contain a real quadratic subfield, rather than an imaginary one.
While these extensions are still deduced from results of Kriz--Li \cite{krizli} by verifying Hypothesis~\eqref{star} for an auxiliary elliptic curve and imaginary quadratic fields, we obtain real quadratic fields where we achieve rank jumps since the minimal discriminant of our chosen elliptic curve is negative (see Lemma~\ref{set Q for two imaginary quadratic one EC negative disc} below for details).

\begin{lemma}
\label{SAGE lemma 704g1}
Consider the elliptic curve $E=$\href{https://www.lmfdb.org/EllipticCurve/Q/704/d/1}{$704g1$}.
\begin{enumerate}[\textup{(}1\textup{)}]
\item The hypotheses of Theorem~\ref{thm:growth-rank} are satisfied.
\item Hypothesis~\eqref{star} holds for $K=\Q(\sqrt{-D})$
where $D$ is in the set\footnote{We only checked through the values of $D$ with $D<700$.}
\[
\mathfrak{D}=\{7, 39, 95, 127, 167, 255, 263, 271, 303, 359, 391, 447, 479, 527, 535, 615, 623, 655, 679, 695\}.
\]
\end{enumerate}
\end{lemma}

\begin{proof}\leavevmode
\begin{enumerate}[\textup{(}1\textup{)}]
\item The hypotheses can be verified using \cite{Sage} (see \S~\ref{S:704GFP}).
\item For $\Q(\sqrt{-7})$ this is recorded in \cite[Table~2]{krizli}.
For the other values of $D$, the conditions can be verified using \cite{Sage}; and the code is provided in \S~\ref{S:star}.
\end{enumerate}
\end{proof}

\begin{lemma}
\label{set Q for two imaginary quadratic one EC negative disc}
Consider the elliptic curve $E=$\href{https://www.lmfdb.org/EllipticCurve/Q/704/d/1}{$704g1$}.
Define the set 
\begin{align*}
\sQ_{D}(E) &= \left\{q\colon q\equiv -1\pmod{4}, \ q \text{ splits in } \Q(\sqrt{-D}), \ a_q(E) \equiv 1\pmod{2} \right\},
\end{align*}
such that Hypothesis~\eqref{star} holds for $\Q(\sqrt{-D})$.
This is a Chebotarev set of primes of density $\frac{1}{12}$.
Moreover, for all $q\in \sQ_{D}$,
\begin{align*}
\rank E\left(\Q(\sqrt{Dq})\right) = 1.
\end{align*}
\end{lemma}

\begin{proof}
The proof is adapted from \cite[Lemma~6.4]{GFP}.
The key difference in the proof is that the elliptic curve \href{https://www.lmfdb.org/EllipticCurve/Q/704/d/1}{$704g1$} has negative minimal discriminant.
Hence, we must apply Corollary~\ref{to be use for the real quadratic version} (instead of Corollary~\ref{basically what GFP used}).
The hypotheses of Corollary~\ref{to be use for the real quadratic version} can be verified directly (for example from LMFDB).
The said corollary asserts that for $d\equiv 1\pmod{4}$ supported on $\sQ_{D}(E)$, we have
\[
\rank E^{(Dd)} = 1.
\]
But note that $d$ can take either positive or negative values.
Therefore, choosing $d$ to be the negative of a prime number in $\sQ_{D}(E)$;
i.e., $d\equiv -q\equiv 1\pmod{4}$, we have that
\[
\rank E^{(Dq)} = 1.
\]
In other words,
\[
\rank E\left(\Q(\sqrt{Dq})\right) = 1.
\]

Since $E=$\href{https://www.lmfdb.org/EllipticCurve/Q/704/d/1}{$704g1$} has surjective mod $2$ representation and $\Q(\sqrt{-11})$ is the unique imaginary quadratic subfield in the Galois extension $\Q(E[2]/\Q)$, the density result follows from Lemma~\ref{GFP 5.1}.
\end{proof}
We can now prove Theorem~\ref{thmB}.

\begin{theorem}
\label{thm: real quadratic one EC two K}
For all $D\in\fD$, there are explicit Chebotarev sets of primes $\sP$ (which is independent of $D$) and $\sQ_D$ of density $\frac{9}{16}$ and $\frac{1}{12}$ respectively such that for all $p\in \sP$ and $q\in \sQ_D$, the analogue of \ref{H10} is unsolvable for the ring of integers of $L=\Q(\sqrt[3]{p},\sqrt{Dq})$.
\end{theorem}

\begin{proof}
This result can be proven in the same way as Theorem~\ref{thm: improve P and Q of GFP}.
The only difference is that we use $E=$\href{https://www.lmfdb.org/EllipticCurve/Q/704/d/1}{$704g1$} and the imaginary quadratic field $\Q(\sqrt{-D})$.
\end{proof}

\section{Studying \texorpdfstring{\ref{H10}}{} Problem via a pair of elliptic curves}
\label{section: pair of ec}

The goal of this section is to study a slightly different version of Theorem~\ref{thm: real quadratic one EC two K} by working with two auxiliary elliptic curves.
This allows us to obtain a larger set of $\sP$, at the expense of a smaller set of $\fD$ with a lower density for $\sQ_D$.

\subsection{Rank stabilization over cubic extensions for a pair of elliptic curves}
\label{set P two curves}

The goal of this section is to study a version of Corollary~\ref{cor:stable-rank-1-curve} for a pair of two elliptic curves (see Corollary~\ref{cor:stable-rank-2-curves} below).
We first introduce the following definition.

\begin{definition}
Let $E_1/\QQ$ and $E_2/\QQ$ be two elliptic curves, and $\ell$ be a prime number.
\begin{itemize}
\item Define the sets 
\begin{align*}
G_{E_1,E_2,\ell}&=\left\{(A,B)\in\GL_2(\ZZ/\ell\ZZ)\times\GL_2(\ZZ/\ell\ZZ) :\det(A)=\det(B) \right\}\\
H_{E_1,E_2,\ell}&=\left\{(A,B)\in G_{E_1,E_2,\ell}:A\in H_{E_1,\ell}\text{ or }B\in H_{E_1,\ell}\right\}.
\end{align*}
\item Elliptic curves $E_1$ and $E_2$ are said to be \textbf{maximally disjoint} at $\ell$ if the image of the product of their mod $\ell$ representations
\[
\jr:G_\QQ\rightarrow \GL(E_1[\ell])\times\GL(E_2[\ell])=\GL_2(\ZZ/\ell\ZZ)\times\GL_2(\ZZ/\ell\ZZ).
\]
is given by $G_{E_1,E_2,\ell}$.
\end{itemize}
\end{definition}
In particular, the definition of \emph{maximally disjoint} implies that both $\bar\rho_{E_1,\ell}$ and $\bar\rho_{E_2,\ell}$ are surjective.
\begin{proposition}
\label{prop:joint3}
Suppose that $E_1/\QQ$ and $E_2/\QQ$ are two elliptic curves that are maximally disjoint at $\ell=3$.
Then,
\[
\frac{\#H_{E_1,E_2,\ell}}{\#G_{E_1,E_2,\ell}}=\frac{103}{128}.
\]
\end{proposition}
\begin{proof}
Under the notation of the proof of Lemma~\ref{lem:proportion}, the conjugacy classes of $\GL_2(\F_3)$ are given by the following representatives
\begin{align*}
  C_{a,b}:& \begin{pmatrix}
  1&\\&-1
  \end{pmatrix}\text{ (12 elements each)}\\
  \cC_a:&\begin{pmatrix}
  1&1\\&1
  \end{pmatrix}, \ 
\ta{a}
{\begin{pmatrix}
  -1&1\\&-1
  \end{pmatrix}}\tb{a}
\text{ (8 elements each)}\\
  D_a:&\begin{pmatrix}
  1&\\&1
  \end{pmatrix},\ \ta{b}{\begin{pmatrix}
  -1&\\&-1
  \end{pmatrix}}\tb{b}\text{ (1 element each)}\\
  \cE_\lambda:&\ta{c}\begin{pmatrix}
  &1\\-1&
  \end{pmatrix}\tb{c}\ ,\ \ta{d}\begin{pmatrix}
  &1\\1&-1
  \end{pmatrix}\tb{d}\ ,\ \ta{e} \begin{pmatrix}
  &1\\1&1
  \end{pmatrix}\tb{e}\text{ (6 elements each)},
\end{align*}
where we have highlighted the ones lying inside $H_{E,\ell}$ in grey shade.
Therefore, we can count the elements in $G_{E_1,E_2,\ell}$ via the following table:\vspace{0.2cm}

\begin{center}

{\tiny
\begin{tabular}{|c|c|c|c|c|c|c|c|c|}
   \hline& $\begin{pmatrix}
  1&\\&-1
  \end{pmatrix}$&$\begin{pmatrix}
  1&1\\&1
  \end{pmatrix}$ &$\begin{pmatrix}
  1&\\&1
  \end{pmatrix}$
   & $\cg{\begin{pmatrix}
  -1&1\\&-1
  \end{pmatrix}}$&$\cg{\begin{pmatrix}
  -1&\\&-1
  \end{pmatrix}}$&$\cg{\begin{pmatrix}
  &1\\-1&
  \end{pmatrix}}$&$\cg{\begin{pmatrix}
  &1\\1&-1
  \end{pmatrix}}$&$\cg\begin{pmatrix}
  &1\\1&1
  \end{pmatrix}$\\\hline

  $\begin{pmatrix}
  1&\\&-1
  \end{pmatrix}$ &144\checkmark&96&12&96&12&72&72\cg{\checkmark}&72\cg\checkmark  \\\hline
  $\begin{pmatrix}
  1&1\\&1
  \end{pmatrix}$&96&64\checkmark&8\checkmark&64\cg\checkmark&8\cg\checkmark&48\cg\checkmark&48&48\\\hline
  $\begin{pmatrix}
  1&\\&1
  \end{pmatrix}$&12&8\checkmark&1\checkmark&8\cg\checkmark&1\cg\checkmark&6\cg\checkmark&6&6\\\hline
   $\cg{\begin{pmatrix}
  -1&1\\&-1
  \end{pmatrix}}$&96&64\cg\checkmark&8\cg\checkmark&64\cg\checkmark&8\cg\checkmark&48\cg\checkmark&48&48\\\hline
  $\cg{\begin{pmatrix}
  -1&\\&-1
  \end{pmatrix}}$&12&8\cg\checkmark&1\cg\checkmark&8\cg\checkmark&1\cg\checkmark&6\cg\checkmark&6&6\\\hline
  $\cg{\begin{pmatrix}
  &1\\-1&
  \end{pmatrix}}$&72&48\cg\checkmark&6\cg\checkmark&48\cg\checkmark&6\cg\checkmark&36\cg\checkmark&36&36\\\hline
  $\cg{\begin{pmatrix}
  &1\\1&-1
  \end{pmatrix}}$&72\cg\checkmark&48&6&48&6&36&36\cg\checkmark&36\cg\checkmark\\\hline
  \cg{$\begin{pmatrix}
  &1\\1&1
  \end{pmatrix}$}&72\cg\checkmark&48&6&48&6&36&36\cg\checkmark&36\cg\checkmark\\\hline
\end{tabular}}\vspace{0.2cm}
\end{center}

\noindent
A check mark signifies that the product of the conjugacy classes belongs to $G_{E_1,E_2,\ell}$, with those lying inside $H_{E_1,E_2,\ell}$ highlighted in grey shade.
We deduce that $\#G_{E_1,E_2,\ell}=1152$ (this is exactly half of $\#\GL_2(\Z/\ell\Z)\times\GL_2(\Z/\ell\Z)$ since $G_{E_1,E_2,\ell}$ is the kernel of the group surjective homomorphism from $\GL_2(\Z/\ell\Z)\times\GL_2(\Z/\ell\Z)$ to $\F_\ell^\times$ given by $(A,B)\mapsto \det(AB)$) and $\#H_{E_1,E_2,\ell}=927$.
\end{proof}

\begin{corollary}
\label{cor:stable-rank-2-curves}
Suppose that  $E_1/\QQ$ and $E_2/\QQ$ are two elliptic curves which are maximally disjoint at $3$ and that both $E_1$ and $E_2$ satisfy the hypotheses of Theorem~\ref{thm:growth-rank} with $\ell=3$.
Then, there exists a Chebotarev set of primes $\sP$ with density $\frac{103}{128}$ such that for all $p\in \sP$, either $\rank E_1\left(\QQ(\sqrt[3]{p})\right)=0$ or $\rank E_2\left(\QQ(\sqrt[3]{p})\right)=0$.
\end{corollary}
\begin{proof}
On taking $\sP$ to be the union $\sP(E_1,3)\cup \sP(E_2,3)$, the result follows from the Chebotarev density theorem, Proposition~\ref{prop:joint3} and Theorem~\ref{thm:growth-rank}.
\end{proof}

Analogously, for an integer $n\ge1$, one may define the set $\sP_n$ as the union $\sP(E_1,3)\cup\ldots\cup \sP(E_n,3)$ for elliptic curves $E_1,\ldots, E_n$ that such that the image of the representation $\bar\rho_{E_1,3}\times\bar\rho_{E_n,3}$ is given by $\left\{(A_1,\cdots, A_n)\in \left(\GL_2(\ZZ/3\ZZ)\right)^n:\det(A_1)=\cdots \det(A_n)\right\}$.
\begin{lemma}
The density for $\sP_n$ is
\[
\frac{2\cdot 24^n-12^n-9^n}{2\cdot 24^n}=1-\frac{3^n+4^n}{2^{3n+1}}.
\]
\end{lemma}
\begin{proof}
There are in total $48^n$ elements in $\left(\GL_2(\Z/\ell\Z)\right)^n$.
Half of the elements of $\GL_2(\Z/\ell\Z)$ have determinant $1$, while the other half have determinant $2$.
Therefore, there are $\frac{1}{2^{n-1}}\times48^n=2\cdot24^n$ $n$-tuples with matching determinant.
Of those, we exclude the ones without factors in $H_{E,\ell}$, i.e., the $n$-tuples of matrices which are
\begin{enumerate}
\item all in the same conjugacy class as $\begin{pmatrix}
  1&\\&-1
  \end{pmatrix}$, of which there are $12^n$, or
\item in all possible combinations of conjugacy classes represented by $\begin{pmatrix}
  1&1\\&1
  \end{pmatrix}$ and $\begin{pmatrix}
  1&\\&1
  \end{pmatrix}$.
\end{enumerate} 
As for the size of the latter type of matrices, the conjugacy class for $\begin{pmatrix}
  1&1\\&1
  \end{pmatrix}$ has size $8$, while that of the identity matrix is $1$, so keeping in mind the possible positions, the latter count is $$8^n+\binom{n}{1}8^{n-1}+\ldots +\binom{n}{n-1}8+1=(8+1)^n.$$ The result follows.
\end{proof}
In particular, the density of $\sP_n$ approaches $1$ as $n\rightarrow \infty$.
In Table~\ref{densities}, we list the densities for $\sP_n$ for some small $n$.
The cases $n=1$ (resp. $n=2$) correspond to Lemma~\ref{lem:proportion} (resp. Corollary \ref{cor:stable-rank-2-curves}).
\begin{center}
\renewcommand{\arraystretch}{1.25}
\setlength{\aboverulesep}{0pt}
\setlength{\belowrulesep}{0pt}
\begin{longtable}{ |c|ccl| }\caption{Densities of $\sP_n$ for $1\le n\le 7$}\label{densities}\\\toprule
$n$ & \multicolumn{3}{l|}{Density for $\sP_n$} \\ \midrule
1 & $\frac{9}{16}$ & $=$ & $0.5625$ \\
2 & $\frac{103}{128}$ & $=$ & $0.8046875$ \\
3 & $\frac{933}{1024}$ & $=$ & $0.9111328\ldots$\\
4 & $\frac{7855}{8192}$ & $=$ & $0.9588623\ldots$ \\
5 & $ \frac{64269}{65536}$ & $=$ & $0.9806671\ldots $ \\
6 & $ \frac{519463}{524288}$ & $=$ & $0.9907970\ldots $\\ 
7 & $ \frac{4175733}{4194304}$ & $=$ & $0.9955723\ldots $\\[3pt] \bottomrule
\end{longtable}
\renewcommand{\arraystretch}{1}
\end{center}

\subsection{Rank jumps for two elliptic curves}
\label{set Q two curves}
Next, we study families of real quadratic fields where the Mordell--Weil ranks of two elliptic curves increase under base-change simultaneously.
\begin{lemma}
\label{proportion of set Q for two ec}
Let $K$ be an imaginary quadratic field.
Suppose that there exist two elliptic curves $E_1/\QQ$ and $E_2/\QQ$ such that their mod $2$ representations are surjective.
Further suppose that $E_1, E_2$ are maximally disjoint at $2$ and that the Galois extension $\Q\left(E_1[2],E_2[2]\right)/\Q$ does not contain $\Q(\sqrt{-1})$ or $K$.
Then, 
\[
\sQ_K(E_1,E_2) = \left\{q\colon q\equiv -1\pmod{4}, \ q \text{ splits in } K, \ a_q(E_1)\equiv a_q(E_2)\equiv 1\pmod{2} \right\}
\]
is a Chebotarev set of primes with density $\frac{1}{36}$.
\end{lemma}

\begin{proof}
We argue as in \cite[Proof of Lemma 5.1(ii)]{GFP}, who treated the case of one elliptic curve and obtained a density of $\frac{1}{6}$.
For $E=E_1$ or $E=E_2$, the condition $a_q(E)\equiv 1 \pmod{2}$ corresponds to the image of $\Frob_q$ being the two elements $\mat{1 & 1 \\ 1 & 0}$ and $\mat{0 & 1 \\ 1 & 1}$, i.e., two of the six possible matrices of $\GL_2(\F_2)\cong S_3$.
Note that they each have determinant $-1$, so the maximally disjoint condition alone would imply that all four combinations of these two matrices occur in the image of $\bar{\rho}_{E_1,E_2,2}$ giving a proportion of $\frac{4}{36}$.
The other conditions cut down the proportion by a factor of $\frac{1}{2}$ each by the Chebotarev density theorem and the fact that $q\equiv -1 \pmod{4}$ is a splitting condition in $\Q(\sqrt{-1})$, which alongside $K$ we assumed to not be in $\Q\left(E_1[2],E_2[2]\right)$.
We thus obtain a density of $\frac{1}{2}\times\frac{1}{2}\times\frac{1}{9}$, i.e., one-third the proportion compared to when we work with one elliptic curve.
\end{proof}

\begin{corollary}\label{applykl}Suppose that there are two elliptic curves $E_1/\Q$ and $E_2/\Q$ satisfying the hypotheses of the above Lemma \ref{proportion of set Q for two ec} and of Corollary \ref{to be use for the real quadratic version} for some imaginary quadratic field $K=\Q(\sqrt{-D})$.
Then for any $q\in \sQ_K(E_1,E_2)$, $\rank E_1(\Q(\sqrt{Dq}))=1=\rank E_2(\Q(\sqrt{Dq}))$.
\end{corollary}

\begin{proof}
We argue as in Corollary \ref{proportion of set Q for two ec} to conclude that the quadratic twists of $E_1$ and $E_2$ by $-D$ have $\Q$-rational points of rank $1$.
Indeed, we can apply the last assertion of Theorem \ref{KL-result} to see that $\rank E(\Q)=\rank E^{(-q)}(\Q)=0$ and from the first assertion, conclude that $\rank E^{((-q)\cdot(-D))}(\Q)=1$ for each of $E=E_1$ and $E=E_2$.
Hence, $\rank E_1(\Q(\sqrt{Dq}))=1=\rank E_2(\Q(\sqrt{Dq}))$.
\end{proof}

This when combined with Corollary~\ref{cor:stable-rank-2-curves} proves Theorem~\ref{thmC}.

\begin{theorem}
\label{thm: real quadratic two EC two K}
There are explicit Chebotarev sets of primes $\sP$ and $\sQ$ (resp. $\sP$ and $\sQ'$) of density $\frac{103}{128}$ and $\frac{1}{36}$ such that for all $(p,q)\in \sP\times\sQ$ (resp. $(p,q')\in \sP\times\sQ'$), the analogue of \ref{H10} is unsolvable for the ring of integers of $L=\Q(\sqrt[3]{p},\sqrt{7q})$ (resp. $L'=\Q(\sqrt[3]{p},\sqrt{615q'})$).
\end{theorem}
\begin{proof}
We combine Corollary \ref{cor:stable-rank-2-curves} and Lemma \ref{proportion of set Q for two ec} with Corollary \ref{applykl}.
We work with the elliptic curves $E_1=$\href{https://www.lmfdb.org/EllipticCurve/Q/704/d/1}{$704g1$} and $E_2=$\href{https://www.lmfdb.org/EllipticCurve/Q/1472/e/1}{$1472j1$}.
Using the code in Appendix~\ref{app:Harris}, we verify maximal disjointedness at $\ell=3$, and check that each elliptic curve satisfies the other required conditions, i.e.,
\begin{enumerate}
\item the hypotheses of Lemma \ref{proportion of set Q for two ec} and 
\item those of Theorem \ref{KL-result}.
\end{enumerate}

As for the former conditions, we know that $E_1$ and $E_2$ are maximally disjoint at $2$, as can be verified in Appendix~\ref{app:Harris}.
One can verify that $\Q(E_1[2],E_2[2])$ contains three degree $2$ subfields of conductors $-11$, $-23$, and $253$.

To check the conditions of Theorem \ref{KL-result},
we refer the reader to the SAGE code in Appendix~\ref{S:704GFP}, where we choose $K=\Q(\sqrt{-7})$ (resp. $K=\Q(\sqrt{-615})$)
We let $\sP=\sP(E_1,3)\cup \sP(E_2,3)$ (defined in Theorem \ref{thm:growth-rank}), and let $\sQ:=\sQ_{\Q(\sqrt{-7})}(E_1,E_2)$ and $\sQ':=\sQ_{\Q(\sqrt{-615})}(E_1,E_2)$ as in Lemma \ref{proportion of set Q for two ec}.
From this we know that for $p \in \sP$ and $q \in \sQ$, either $\rank E_1\left(\QQ(\sqrt[3]{p})\right)=0$ or $\rank E_2\left(\QQ(\sqrt[3]{p})\right)=0$ from Corollary~\ref{cor:stable-rank-1-curve}, and for $q \in \sQ$, we have $\rank E_1(\Q(\sqrt{Dq}))=\rank E_2(\Q(\sqrt{Dq}))=1$ for $D=7$ and for $D=615$.
We can now apply Proposition \ref{prop: GFP 3.3} to at least one of $E_1$ or $E_2$.
The result follows then from \ref{Thm:Shlapentokh}.
\end{proof}

\begin{remark}\label{Remark:future}
We included only two elliptic curves, obtaining the density corresponding to $n=2$ in Table \ref{densities}.
With enough computing power, it should in principle be possible to generalize the Theorem~\ref{thm: real quadratic two EC two K} further using  $n$ auxiliary elliptic curves satisfying the appropriate hypotheses for an arbitrary $n$.
It would be interesting to see how big $n$ could be made with the most powerful computers available.
The analogue of the set $\sP$ should then have density $1-\frac{3^n+4^n}{2^{3n+1}}$ (cf. Table \ref{densities}, where the analogue of $\sP$ is called $\sP_n$). As for the analogue of the set $\sQ$ built out of $n$ curves,
 the "worst case scenario" would occur when the image of the representation $\bar\rho_{E_1,\ldots, E_n,2}=\bar\rho_{E_1,2}\times\cdots\times \bar\rho_{E_n,2}$ is $\left(\GL_2(\ZZ/2\ZZ)\right)^n$. Its density  would be $ \frac{1}{4}\times\frac{1}{3^n}$.
However, it is possible the density is higher and does not converge to zero.
For example, if the image of $\bar\rho_{E_1,\ldots, E_n,2}$ is isomorphic to the diagonal embedding of $\GL_2(\ZZ/2\ZZ)$, then the  density would be $\frac{1}{12}$.
\end{remark}

\begin{remark}
Suppose that one could find an elliptic curve other than  \href{https://www.lmfdb.org/EllipticCurve/Q/557/b/1}{$557b1$}, which has positive discriminant and satisfies the hypotheses of Theorem~\ref{thm:growth-rank} with $\ell=3$.
If furthermore the new curve and \href{https://www.lmfdb.org/EllipticCurve/Q/557/b/1}{$557b1$} are maximally disjoint at $3$, then one could obtain a similar result to Theorem~\ref{thm: real quadratic two EC two K}, where one improves the density of $\sP$ in Theorem~\ref{thm: improve P and Q of GFP} to $\frac{103}{128}$.
The density of the resulting $\sQ$ would depend on the joint image of the mod $2$ representations of the two curves.
Due to the limited computing power we have had access to, we have not been able to find such a curve.
\end{remark}
\section{Congruent number elliptic curves and \texorpdfstring{\ref{H10}}{}}
\label{section: Congruent Number Curves}
We study a different version of Theorem~\ref{thm: real quadratic two EC two K} using congruent number elliptic curves.
One advantage of this approach is that one may make use of the  recent breakthroughs on Goldfeld's conjecture for these curves to study the set $\sQ$.

\subsection{Congruent number elliptic curves}

In this section we study the congruent number elliptic curve $E=\href{https://www.lmfdb.org/EllipticCurve/Q/32/a/3}{32a2}$, defined by the equation $y^2=x^3-x$.
Note that $E$ has good supersingular reduction at $3$ with $a_3(E)=0$.

Similar to the curve \href{https://www.lmfdb.org/EllipticCurve/Q/704/d/1}{704g1} studied in Lemma~\ref{SAGE lemma 704g1}, we can check that $E$ satisfies the hypotheses of Theorem~\ref{thm:growth-rank} for $\ell=3$ using \cite{Sage}.

We now calculate the density of the set $\sP(E,3)$ given in Theorem~\ref{thm:growth-rank}.

\begin{lemma}\label{lem:H/G}
For $E=\href{https://www.lmfdb.org/EllipticCurve/Q/32/a/3}{32a2}$, we have
\[
\frac{\# H_{E,3}}{\# G_{E,3}}=\frac{11}{16}.
\]
In other words, the set $\sP(E,3)$ defined in the statement of Theorem~\ref{thm:growth-rank} is a Chebotarev set of density $\frac{11}{16}$.
\end{lemma}
\begin{proof}
The curve $E$ has complex multiplication and the image of $\bar\rho_{E,3}$ is maximal, given by the following 16 matrices:
\begin{align*}
  \begin{pmatrix}
  1&\\&1
  \end{pmatrix},
  \begin{pmatrix}
  2&2\\&1
  \end{pmatrix},
  \begin{pmatrix}
  1&1\\&2
  \end{pmatrix},
  \begin{pmatrix}
  2&\\2&1
  \end{pmatrix},
  \begin{pmatrix}
  1&\\1&2
  \end{pmatrix}, \text{ and }
  \ta{1}\begin{pmatrix}
  2&\\&2
  \end{pmatrix},
  \begin{pmatrix}
  2&2\\2&1
  \end{pmatrix},
  \begin{pmatrix}
  1&1\\1&2
  \end{pmatrix},\\
  \ta{2}\begin{pmatrix}
  2&1\\1&
  \end{pmatrix},
  \begin{pmatrix}
  1&2\\2&
  \end{pmatrix},
  \begin{pmatrix}
  &2\\2&2
  \end{pmatrix},
  \begin{pmatrix}
  &1\\1&1
  \end{pmatrix},
  \begin{pmatrix}
  &2\\1&
  \end{pmatrix},
  \begin{pmatrix}
  &1\\2&
  \end{pmatrix},
  \begin{pmatrix}
  1&2\\2&2
  \end{pmatrix},
  \begin{pmatrix}
  2&1\\1&1
   \end{pmatrix}\tb{1}\tb{2},
\end{align*}
with the ones lying inside $H_{E,3}$ highlighted in grey shade.
\end{proof}

\begin{remark}
Curiously, even though $\bar\rho_{E,3}$ is not surjective, we are actually getting a larger proportion than the case where the representation is surjective studied in Lemma~\ref{lem:proportion}.
\end{remark}

A straightforward application of Theorem~\ref{thm:growth-rank} shows that
\begin{corollary}
\label{cor:Pcong}
There exists a Chebotarev set $\sP_\con$ of primes of density $\frac{11}{16}$ such that the Mordell--Weil rank of $E/\QQ(\sqrt[3]{p})$ is $0$ for all $p\in\sP_\con$.
\end{corollary}

\subsection{Rank jump in real quadratic fields}
We now turn our attention to studying real quadratic fields over which $E$ has positive rank.
The congruent number problem predicts that for all positive square-free integers $n$ that are congruent to $5,6,7\mod 8$, the quadratic twist of $E$ given by 
\[
E^{(n)}:ny^2=x^3-x
\]
has positive Mordell--Weil rank.
In particular, it predicts that $\rank E\left(\QQ(\sqrt{n})\right)$ is positive.

\begin{definition}\label{def:Qcong}
Let $\sQ_\con$ be the set of primes $q$ such that the elliptic curve $qy^2=x^3-x$ has positive Mordell--Weil rank over $\QQ$ (or equivalently $q$ is a congruent number).
\end{definition}

A result announced by Smith in \cite[Theorem~1.5]{smith} says that $E^{(n)}$ has positive Mordell--Weil over $\QQ$ for at least $62.9\%$ of $n$ that are congruent to $5$ and $7$ modulo $8$.
In particular, it says that $\sQ_\con$ has density at least $31.45\%$.
More recently, Kriz \cite{kriz} announced a proof of Goldfeld's conjecture for the family of congruent number elliptic curves.
In particular, it says that the density of $\sQ_\con$ is precisely $\frac{1}{2}$.

We conclude this section with the following result, which is Theorem~\ref{thmD} in the introduction.

\begin{theorem}\label{thm:cong}
Let $\sP_\con$ and $\sQ_\con$ be the set of primes defined as in Corollary~\ref{cor:Pcong} and Definition~\ref{def:Qcong}.
Then, the analogue of \ref{H10} is unsolvable for the ring of integers of $L=\Q(\sqrt[3]{p},\sqrt{q})$.
\end{theorem}
\begin{proof}
This follows from the same proof as Theorem~\ref{thm: improve P and Q of GFP}.
\end{proof}

\appendix
\section{MAGMA code for verification of maximal disjointedness}\label{app:Harris}
The Magma code below, written by Harris B.~Daniels, is used to verify whether two elliptic curves are maximally disjoint at a prime $p$.

{\small
\begin{verbatim}
    function SimpleSplit(f)
    Factors := [vec[1] : vec in Factorization(f) | Degree(vec[1]) gt 1];
    Fields := [NumberField(fac) : fac in Factors];
    K := Rationals();
    for F in Fields do
        K := Compositum(F,K);
    end for;
    return K;
end function;

//Given 2 elliptic curves E1 and E2 and a prime p, it returns true if Q(E1[p]) meet Q(E2[p]) eq Q(zeta_p).
//Note that if Q(E1[p]) eq Q(E2[p]) eq Q(zeta_p), then this function will return true.

function HasMaxDisjointPTorsion(E1,E2,p)
    f1 := DivisionPolynomial(E1,p);
    f2 := DivisionPolynomial(E2,p);
    K1<a1> := SimpleSplit(f1);
    K2<a2> := SimpleSplit(f2);
    P1<y1> := PolynomialRing(K1);
    P2<y2> := PolynomialRing(K2);

//Now we see if we need to add any y coordinates.
    rts1 := [vec[1] : vec in Roots(P1!f1) | IsIrreducible(Evaluate(DefiningPolynomial(E1),[vec[1],y1,1]))];
    rts2 := [vec[1] : vec in Roots(P2!f2) | IsIrreducible(Evaluate(DefiningPolynomial(E2),[vec[1],y2,1]))];
//If there are y-coordinates to add, we add them now to K1 and K2 to make L1 and L2.
    if #rts1 ne 0 then
        p1 := Evaluate(DefiningPolynomial(E1),[rts1[1],y1,1]);
        L1<b1> := AbsoluteField(ext<K1 | p1>);
    else
        L1<b1> := K1;
    end if;

    if #rts2 ne 0 then
        p2 := Evaluate(DefiningPolynomial(E2),[rts2[1],y2,1]);
        L2<b2> := AbsoluteField(ext<K2 | p2>);
    else
        L2<b2> := K2;
    end if;

    dp1 := DefiningPolynomial(L1); //The splitting field of L1 is Q(E1[p])
    dp2 := DefiningPolynomial(L2); //The splitting field of L2 is Q(E2[p])

    //This last step checks the degrees are what they should be.
    s1 := #GaloisGroup(dp1);
    s2 := #GaloisGroup(dp2);
    s3 := #GaloisGroup(dp1*dp2);
    return s1*s2 div EulerPhi(p) eq s3;
end function;

E1 := EllipticCurve("704g1");
E2 := EllipticCurve("1472j1");
p := 3; // change the 3 to a 2 to check for maximal disjointness at 2


HasMaxDisjointPTorsion(E1,E2,p);
\end{verbatim}
}

\section{SAGE code for preserving the rank in cubic extensions}

\subsection{Elliptic curve of conductor \texorpdfstring{$704$}{}}\label{S:704GFP}
Start with the elliptic curve $E = $ \href{https://www.lmfdb.org/EllipticCurve/Q/704/d/1}{$704g1$} (Cremona label).
\begin{enumerate}
\item $E$ has good ordinary reduction at $3$.
\item At all the primes $p$, the residual Galois representation of $E$ has maximal image.
\item We also check that $3$ is not a prime of anomalous reduction.
\item The twist of $E$ by $-3$ has rank $0$.
To calculate the rank of the twisted elliptic curve use:
\begin{verbatim}
E = EllipticCurve('704g1')
Et = E.quadratic_twist(-3)
Et.rank()
\end{verbatim}
\item It is clear that the base-change of $E$ to $\Q(\sqrt{-3})$ has rank 0.
Just to be sure, we can check:
\begin{verbatim}
E = EllipticCurve('704g1')
K.<t> = QuadraticField(-3)
EK = E.base_extend(K)
EK.rank()
\end{verbatim}
\item We check that the order of the torsion group of $E$ over $\Q(\sqrt{-3})$ is not divisible by $3$.
Just to be sure, we can check:
\begin{verbatim}
E = EllipticCurve('704g1')
K.<t> = QuadraticField(-3)
EK = E.base_extend(K)
EK.torsion_group()
\end{verbatim}
\item The primes of bad reduction are $2, 11$ (which are both inert in $\Q(\sqrt{-3})$).
This is also clear from \cite[Table~2]{krizli}.
To obtain the local data use: 
\begin{verbatim}
E = EllipticCurve('704e1')
K.<t> = QuadraticField(-3)
EK = E.base_extend(K)
EK.local_data(2)    
EK.local_data(11)
\end{verbatim}
Output: The Tamagawa number at the inert prime $2$ is $1$ and at the inert prime $11$ is $1$.
\item To see that $3\nmid \#\Sha(E/\Q(\sqrt{-3}))$, we can use \cite[Theorem~2.1]{Qiu14}.
It suffices to check that $\#\Sha(E^{(-3)}/\Q)$ is not divisible by $3$.
For this, use:
\begin{verbatim}
E = EllipticCurve('704g1')
Et = E.quadratic_twist(-3)
St = Et.sha()
St.an()    
\end{verbatim}
Output: The (analytic) Shafarevich--Tate group is trivial for the quadratic twist of $E$.
\end{enumerate}

\begin{remark}
This elliptic curve satisfies the Hypothesis \eqref{star} when the imaginary quadratic field is $\Q(\sqrt{-7})$, see \cite[Table~2]{krizli}.
\end{remark}

\subsection{Elliptic curve of conductor \texorpdfstring{$1472$}{}}
Now we work with the elliptic curve $E = $ \href{https://www.lmfdb.org/EllipticCurve/Q/1472/e/1}{$1472j1$} (Cremona label).
For this curve, we can check all the properties as in the previous section.
This elliptic curve is not in the table of Kriz--Li since the conductor is larger than $750$.
We now check that Hypothesis \eqref{star} is satisfied in Appedix~\ref{app: Hyp star}.

\section{SAGE code for verifying \texorpdfstring{Hypothesis~\eqref{star}}{} for curves \texorpdfstring{\href{https://www.lmfdb.org/EllipticCurve/Q/704/d/1}{$704g1$}}{},  \texorpdfstring{\href{https://www.lmfdb.org/EllipticCurve/Q/1472/e/1}{$1472j1$}}{}, and \texorpdfstring{\href{https://www.lmfdb.org/EllipticCurve/Q/557/b/1}{$557b1$}}{}}
\label{app: Hyp star}

\subsection{When conductor is  \texorpdfstring{$704$ or $1472$}{}}\label{S:star}
The prime $2$ is a prime of bad reduction for the curve of conductor $704$ and $1472$.
Since the Heegner hypothesis is satisfied for $K=\Q(\sqrt{-7})$, it follows immediately that $2$ splits in the extension $K/\Q$.
The other way to check is to use quadratic reciprocity: since $-7\equiv 1\pmod{4}$ it follows that the discriminant $d_K$ of $K$ is $-7$.
Since, $d_K \equiv 1\pmod{8}$, we know that $2$ splits in $K$.
The same argument shows that $2$ splits in $K=\sqrt{-615}$.

\subsubsection*{Finding the size of \texorpdfstring{$\tilde{E}^{ns}(\bF_2)$}{}} Note that $2$ is a prime of additive reduction for both the elliptic curves of interest.
Hence, 
\[
\abs{\tilde{E}^{ns}(\bF_2)} = \abs{\tilde{E}(\bF_2)}-1 = 3-1 = 2.
\]

\subsection*{Finding the Heegner Point in \texorpdfstring{$K$}{}}

\subsubsection*{When \texorpdfstring{$K$}{} has class number \texorpdfstring{$1$}{}}
Since $\Q(\sqrt{-7})$ has class number $1$, the Heegner point lies in $K$ itself.
However, we write a general code suggested by J.~Cremona available \href{https://groups.google.com/g/sage-nt/c/t591aIzT2_s?pli=1}{here} which will allow us to change $K$:
\begin{verbatim}
E = EllipticCurve('704g1') 
a = -7 # conductor of the imaginary quadratic field 
P = E.heegner_point(a) # this Heegner point is in H(K)
P1 = P.point_exact() 
K = P1[0].parent()
E = P1.curve()
G = K.automorphisms()
def apply(sigma, pt):
     E = pt.curve()
     return E([sigma(c) for c in pt])
sum([apply(sigma,P1) for sigma in G], 0)
s = K(a).sqrt()
H = sum([apply(sigma,P1) for sigma in G if sigma(s)==s], 0) 
# this is the Heegner point in K
r1 = -H[0]/-H[1] # r= -x(H)/y(H) #parameter corresponding to H
H2 = H + H # sum of Heegner points 2H
r2 = -H2[0]/-H2[1] # r= -x(H2)/y(H2) #parameter corresponding to 2H
\end{verbatim}

\subsection*{Calculating the \texorpdfstring{$2$}{}-adic valuation of \texorpdfstring{$\abs{\widetilde{E}^{\ns}(\F_2)} \log_{\omega_E}(P)$}{}}

\begin{verbatim}
L2=r2.parent()
u=QQ.valuation(2)
vL2=u.extensions(L2); v2=vL2[0]
print(v2(r2), v2(r2-r2^3/3)) 
# obtain what to evaluate by checking the first few terms of E.formal_group().log(10)
\end{verbatim}
The output is $1 \ 1$.
This means that Hypothesis \eqref{star} is satisfied for $K=\Q(\sqrt{-7})$.

\subsubsection*{When \texorpdfstring{$K$}{} does not have class number \texorpdfstring{$1$}{}}
Our code above works for any imaginary quadratic field, not just those with class number $1$.
We use it to check that both the elliptic curves of interest, namely \href{https://www.lmfdb.org/EllipticCurve/Q/704/d/1}{$704g1$} and \href{https://www.lmfdb.org/EllipticCurve/Q/1472/e/1}{$1472j1$} satisfy Hypothesis \eqref{star} when the imaginary quadratic field is $\Q(\sqrt{-615})$.

\subsection{When conductor is  \texorpdfstring{$557$}{}}
In \cite[Table~2]{krizli}, it is already checked that $E$=\href{https://www.lmfdb.org/EllipticCurve/Q/557/b/1}{$557b1$} satisfies Hypothesis \eqref{star} when the imaginary quadratic field is $\Q(\sqrt{-7})$.
To check that the same holds for $\Q(\sqrt{-79})$ (resp. $\Q(\sqrt{-127})$), we need to first verify that $2$ splits in $\Q(\sqrt{-79})$ (resp. $\Q(\sqrt{-127})$).
This follows from the observation that $-127 \equiv 1\pmod{8}$.

For the elliptic curve
$E=\href{https://www.lmfdb.org/EllipticCurve/Q/557/b/1}{557b1}$, the prime $2$ is a prime of good reduction so we can use the code to know $\abs{\widetilde{E}^{\ns}(\F_2)}$:

\begin{verbatim}
EllipticCurve(GF(2),'557b1').cardinality()
\end{verbatim}
The output is 1.

Now we calculate the $2$-adic valuation of $\abs{\widetilde{E}^{\ns}(\F_2)} \log_{\omega_E}(P) = \log_{\omega_{E}}(P)$:

\begin{verbatim}
E = EllipticCurve('557b1') 
a = -127 # replace with -79 for the other verification
P = E.heegner_point(a) 
P1 = P.point_exact(300) #might have to increase precision
K = P1[0].parent()
E = P1.curve()
G = K.automorphisms()
def apply(sigma, pt):
     E = pt.curve()
     return E([sigma(c) for c in pt])
sum([apply(sigma,P1) for sigma in G], 0)
s = K(a).sqrt()
H = sum([apply(sigma,P1) for sigma in G if sigma(s)==s], 0) 
# this is the Heegner point in K
r1 = -H[0]/-H[1] # r= -x(H)/y(H) #parameter corresponding to H
L1=r1.parent()
u=QQ.valuation(2)
vL1=u.extensions(L1); v1=vL1[0]
print(v1(r1-r1^3/3))
\end{verbatim}
The output is $1$.
This means that Hypothesis \eqref{star} is satisfied for $K=\Q(\sqrt{-127})$.

\bibliographystyle{amsalpha}
\bibliography{references}
\end{document}